\documentclass[12pt]{article}

\usepackage[english]{babel}

\usepackage{amssymb}
\usepackage{amsfonts}
\usepackage{cite}
\usepackage{dsfont}

\usepackage{amsmath,amsthm, epsf}
\usepackage{pgf,pgfarrows,pgfnodes,pgfautomata,pgfheaps}

\usepackage{wrapfig}

\usepackage{graphicx}
\usepackage{caption}

\newcommand{\ep}{\varepsilon}
\newcommand{\dist}{\mbox{dist}}

\newtheorem{theorem}{Theorem}
\newtheorem{remark}{Remark}
\newtheorem{lemma}{Lemma}

\title{On absolute nonshadowability of transitive maps}
\author{Sergey Tikhomirov\footnote{Max Planck Institute for Mathematics in the Science Inselstrasse 22, 04103 Leipzig, Germany.
Chebyshev Laboratory, Saint-Petersburg State University
14th lane 29B, Vasilievsky Island, St. Petersburg, 199178, Russia. email:sergey.tikhomirov@gmail.com}}
\date{}

\begin{document}

\maketitle
\begin{abstract}
We study shadowing property for random infinite pseudotrajectories of a continuous map $f$ of a compact metric space. For the cases of transitive maps and transitive attractors we prove a dichotomy: either $f$ satisfies shadowing property or random pseudotrajectory is shadowable with probability 0.
\end{abstract}


\section{Introduction}

The theory of shadowing of approximate trajectories
(pseudotrajectories) of dynamical systems is now a well-developed
part of the global theory of dynamical systems (see
the monographs \cite{PilBook, PalmBook} and \cite{PilRev} for a survey of modern results). The shadowing problem is related to the following question: under which conditions, for any pseudotrajectory of $f$ does there exist a close trajectory?

It is known that a diffeomorphism $f$ has the shadowing property in a neighborhood of a hyperbolic set \cite{Ano,Bow}. Moreover if $f$ is structurally stable (see definition for example in \cite{Katok,PilSpaces}) then it has the shadowing property on the whole manifold. At the same time, it is easy to give an example of a diffeomorphism that is not structurally stable but has the shadowing property (see \cite{TikhHol}, for instance). Thus, structural stability is not equivalent to shadowing.

At the same time it was proved that in several contexts shadowing and structural stability are equivalent.
Sakai prove that the $C^1$-interior of the set of diffeomorphisms having the shadowing property coincides with the set of structurally stable diffeomorphisms \cite{Sak} (see \cite{PilRodSak} for a similar result for the orbital shadowing property). Abdenur and Diaz conjectured that a $C^1$-generic diffeomorphism with the shadowing property is structurally stable;
they have proved this conjecture for the so-called tame diffeomorphisms~\cite{AbdDiaz}. Jointly with S. Pilyugin the author proved that so-called Lipschitz shadowing is equivalents to structural stability \cite{PilTikhLipSh}. Thus, set of not structurally stable diffeomorphisms satisfying shadowing property is not very reach.

It is a natural problem to find a shadowing property which is satisfied for a broader class of diffeomorphisms. One of the possible approaches is to consider random pseudotrajectories: endow the space of pseudotrajectories with a probability measure and find sufficient conditions for probability of a pseudotrajectory to be shadowable to be close to 1 or at least positive.

Such studies were initiated in \cite{YY00}. In this work Yuan, Yorke constructed an open set of diffeomorphisms for which probability of a pseudotrajectory to be shadowable is 0. In a recent work \cite{TikhSkew} the author considered a special example of linear skew product and found probability of a finite pseudotrajectory to be shadowable.

Despite the naturalness of randomness approach currently consideration of finite pseudotrajectories is more developed \cite{Yorke1, Yorke2, TikhHol, TikhSkew}. One of the reasons is lack of positive results about shadowing of random pseudotrajectories.

In the present paper we prove that such a positive result is not possible: under transitivity assumption either all pseudotrajectories are shadowable or probability of a pseudotrajectory to be shadowable is 0. Precise statements of the results for transitive maps and transitive attractors are formulated in Theorems \ref{thmTransitive}, \ref{thmAttractor} respectively.


\section{Transitive maps}


Let $(M, \dist)$ be a compact metric space  endowed with a finite Borel measure $\mu$, such that for any open set $U$ the inequality $\mu(U) > 0$ holds. For $a> 0$, $x \in M$ denote by $B(a, x)$ the open ball of radius $a$ centered at $x$. Let $f: M \to M$ be a homeomorphism.

Let $I = [0, N]$ or $I = [0, +\infty)$. For $d>0$ we say that sequence $\{y_n\}_{n \in I}$ is a \textit{$d$-pseudotrajectory} if
\begin{equation}\label{eqpst}
\dist(y_{n+1}, f(y_n)) \leq d, \quad n, n+1 \in I.
\end{equation}
We say that a $d$-psedotrajectory $\{y_n\}_{n \in I}$ can be \textit{$\ep$-shadowed} for $\ep > 0$ if there exists $x_0 \in M$ such that
\begin{equation}\label{eqsh}
\dist(y_n, f^n(x_0)) \leq \ep, \quad n \in I.
\end{equation}
We say that $f$ has \textit{shadowing property} if for any $\varepsilon > 0$ there exists $d > 0$ such that any $d$-pseudotrajectory $\{y_n\}_{n \geq 0}$ can be $\ep$-shadowed.

\begin{remark}
In the definitions of pseudotrajectories and shadowing in equations \eqref{eqpst}, \eqref{eqsh} usually are used strict inequalities. The definition of the shadowing property with not strict inequalities is equivalent to the classical one and allows us to simplify the notation in the proofs of main results.
\end{remark}



For $y_0 \in M$, $d> 0$, $N \in \mathds{N} \cup \{+\infty\}$ denote by $\Omega_N(y_0, d)$ the set of all $d$-pseudotrajectories $\{y_n\}_{n = 0}^{+\infty}$ starting at $y_0$. Let us consider point $y_{n+1}$ being chosen at random in $B(d, f(y_n))$ uniformly with respect to the measure $\mu$: for a measurable set $A \subset M$ the equality
$$
P(y_{n+1}\in A|y_n) = \frac{\mu(A\cap B(d, f(y_n)))}{\mu(B(d, f(y_n)))}
$$
holds. Then $\Omega_N(y_0, d)$ forms a Markov chain. This naturally endows $\Omega_N(y_0, d)$ with a probability measure $P^{y_0, d}_N$. For simplicity for $N=+\infty$ we will omit it: $\Omega(y_0, d)$, $P^{y_0, d}$.

\begin{remark}
This concept was introduced in \cite{YY00} for infinite pseudotrajectories, see also \cite{TikhSkew} for a similar concept for finite pseudotrajectories.
\end{remark}

For any $y_0 \in M$, $d > 0$, $\ep > 0$, $N \in \mathds{N} \cup \{+\infty\}$ consider set $Sh_N(y_0, d, \ep) \subset \Omega_N(y_0, d)$ of pseudotrajectories $\{y_n\}$ which can be $\ep$-shadowed.
For $N = +\infty$ we denote the set as $Sh(y_0, d, \ep)$. Note that each of the sets $Sh_N(y_0, d, \ep)$ is closed for $N \in \mathds{N}$ and is measurable with respect to $P^{y_0, d}_N$. Hence $Sh(y_0, d, \ep)$ is a countable intersection of measurable events and is measurable itself.

Denote
$$
p(y_0, d, \ep) = P^{y_0, d}(Sh(y_0, d, \ep))
$$
the probability that a $d$-pseudotrajectory starting at $y_0$ can be $\ep$-shadowed.

We say that a map $f$ is \textit{absolutely not shadowable} if there exists $\varepsilon > 0$ such that for any $y_0 \in M$, $d> 0$, $p(y_0, d, \ep) = 0$. Speaking informally almost any trajectory of $f$ cannot be $\ep$-shadowed.
In \cite{YY00} Yuan and Yorke provided a class of examples of absolutely not shadowable diffeomorphism. Based on the same technique Abdenur and Diaz proved that absolute nonshadowability is $C^1$ generic among not structurally stable maps.


We say that $f$ is \textit{transitive} if there exists $r \in M$ such that
\begin{equation}\label{eqTrans}
\overline{O^+(r, f)} = M.
\end{equation}

In \cite{AbdDiaz} Abdenur and Diaz proved that $C^1$-robustly transitive and not hyperbolic diffeomorphisms  are absolutely not shadowable. Their proof is based on construction of periodic orbits with different indices.

In the present paper we remove the differentiability assumption, and what is more important, do not assume any properties of the perturbation of $f$.

\begin{theorem}\label{thmTransitive}
  If $f$ is a transitive map then one of the following holds
  \begin{itemize}
    \item[(i)] $f$ has the shadowing property;
    \item[(ii)] $f$ is absolutely not shadowable.
  \end{itemize}
\end{theorem}

In the proof we will use the following folklore result.
\begin{lemma}\label{lemShFin}
If for any $\ep > 0$ there exists $d > 0$ such that any finite $d$-pseudotrajectory $\{y_n\}_{n = 0}^N$, $N > 0$  can be $\ep$-shadowed then $f$ satisfies the shadowing property.
\end{lemma}

\begin{proof}[Proof of Theorem \ref{thmTransitive}]
Assume that $f$ does not satisfy shadowing property.

Lemma \ref{lemShFin} implies that there exists $\ep > 0$ such that for any $d > 0$ there exists $N$ and a $d/2$-pseudotrajectory $\{p_n\}_{n = 0}^N$ which cannot be $\ep$-shadowed.

Take this  $\ep> 0$ and fix arbitrarily $d > 0$ and corresponding $N$ and $\{p_n\}_{n = 0}^N$ a not $\ep$-shadowable $d/2$-pseudotrajectory. Below we will show that
\begin{equation}\label{eq3.1}
p(y_0, d, \ep/2) = 0.
\end{equation}
Let us note that any sequence $\{z_n\}_{n = 0}^N$ satisfying
$$
\dist(p_n, z_n) < \ep/2
$$
is not $\ep/2$-shadowable.

Since $f$ is continuous there exists $\delta > 0$ such that for any $x \in M$, $y \in B(d/2, f(x))$, $z \in B(\delta, x)$ the inclusion
\begin{equation}\label{eqAdd4.1}
B(\delta, y) \subset B(d, f(z))
\end{equation}
holds.

Denote
\begin{equation}\label{eqAdd4.2}
\eta = \frac{\inf_{x\in M}\mu(B(\delta, x))}{\sup_{x\in M} \mu(B(d, x))} \ne 0,
\end{equation}
$$
\mathcal{B} = \{\{q_n\}_{n = 0}^N: \; q_n \in B(\delta, p_n)\}.
$$
Let us show that for any $z_0 \in B(\delta, p_0)$ the inequality
\begin{equation}\label{eq4.1}
  P_N(\Omega_N(z_0, d) \cap \mathcal{B}) \geq \eta^N
\end{equation}
holds. Indeed, let $z_0, z_1, \dots, z_N$ be a random $d$-pseudotrajectory. Then for any $z_k \in B(\delta, p_k)$ the inequality
$$
P_N(z_{k+1} \in B(\delta, p_k)|z_k) \geq \frac{\mu(B(\delta, p_{k+1}))}{\mu(B(d, f(z_k)))}\geq \eta
$$
holds. Multiplying those inequalities for $k = 0, \dots, N-1$ and using Markov property we conclude that
$$
P_N(z_{k+1} \in B(\delta, p_{k+1}), \; k = 0, \dots, N-1) \geq \eta^N,
$$
which proves \eqref{eq4.1}.

Let us consider finite covering $\{U_i\}$ of $M$ by open balls of radius $\delta_1 = \delta/4$. Let $r\in M$ satisfies \eqref{eqTrans}. Trajectory of point $r$ visits each of $\{U_i\}$ infinitely many times. Let $K_1$ be such that $\{f^n(r)\}_{n = 0}^{K_1}$ visits each of $\{U_i\}$ and $K_2$ be such that $\{f^n(f^{K_1+1}(r))\}_{n = 0}^{K_2}$ visits each of $\{U_i\}$. Set $K = K_1 + K_2 +1$. For any $z_0 \in M$ there exists $0 \leq n_1 \leq n_2 \leq K$ such that
\begin{equation}\label{eq6.1}
\dist(z_0, f^{n_1}(r)) < 2 \delta_1, \quad \dist(p_0, f^{n_2}(r)) < 2 \delta_1.
\end{equation}

Consider sequence
$$
\{q_n\}_{n = 0}^{n_2 - n_1 + N} = \{f^{n_1}(r), \dots, f^{n_2 - 1}(r), p_0, p_1, \dots, p_N\}.
$$
Due to inequalities \eqref{eq6.1} sequence $\{q_n\}$ is a $d/2$-pseudotrajectory. Since it contains $\{p_n\}_{n = 0}^N$ it cannot be $\ep/2$-shadowed. Similarly to \eqref{eq4.1}
$$
P^{z_0, d}(z_n \in B(\delta, q_n), n \in [0, n_2 - n_1 + N]) \geq \eta^{n_2 - n_1 + N} \geq \eta^{K+N}.
$$
Denote $L = K+N+1$. Hence for any $z_0$
$$
P^{z_0, d}(\{z_n\}_{n = 0}^{L-1} \mbox{ is not $\ep/2$-shadowable}|z_0) \geq \eta^{L-1} \geq \eta^{L}.
$$
Similarly for any $k \geq  0$ and $y_{kL} \in M$
$$
P^{y_0, d}(\{y_n\}_{n = kL}^{(k+1)L-1} \mbox{ is not $\ep/2$-shadowable}|y_{kL}) \geq \eta^L.
$$
Combining those inequalities we conclude that for any $k \geq 0$
$$
P^{y_0, d}(\{y_n\}_{n = 0}^{(k+1)L-1} \mbox{ is not $\ep/2$-shadowable}) \geq 1- (1-\eta^L)^k.
$$
The right-hand side of the letter expression tends to 1 as $k \to \infty$ hence
$$
P^{y_0, d}(\{y_n\}  \mbox{ is not $\ep/2$-shadowable}) = 1.
$$
Theorem is proved.
\end{proof}

\section{Transitive attractors}

We say that an invariant compact set $\Lambda$ is \textit{an attractor} if there exists an open neighborhood $U$ of $\Lambda$ such that $f(U) \subset U$ and $\cap_{n \geq 0}f^n(U) = A$, see for instance \cite{KatokFirst}. See book \cite{Leonov} for systematic studies of properties of attractors.

We will use the following two properties of an attractor:
\begin{enumerate}
  \item $\dist(f^n(x), \Lambda) \to 0$ as $n \to \infty$ for any $x \in U$;
  \item for any neighborhood $V \subset U$ of $\Lambda$ there exists a neighborhood $W \subset V$ such that
      \begin{equation}\label{eqAdd6.1}
          \overline{f(W)} \subset W.
      \end{equation}
\end{enumerate}

Denote by $D(\Lambda)$ the domain of attraction of $\Lambda$:
$$
D(\Lambda) := \{x \in M: \; \dist(f^n(x), \Lambda) \to 0 \}.
$$
Note that $D(\Lambda) = \cup_{n \geq 0} f^{-n}(U)$.

We say that $f$ has the shadowing property on $\Lambda$ if for any $\ep > 0$ there exists $d > 0$ such that for any $d$-pseudotrajectory $\{y_n\}_{n \geq 0} \subset \Lambda$ there exists $x_0 \in M$ (not necessarily belonging to $\Lambda$)
such that inequalities \eqref{eqsh} hold.

We say that set $\Lambda$ is transitive if there exists $r \in \Lambda$ such that
$$
\Lambda = \overline{O^+(r, f)}.
$$

The assumption for a map to be transitive is quite restrictive. At the same time it is quite common for attractors. In some works transitivity  is included in the definition of attractor \cite{AbdNonlinearity}. In this work it was proved that transitive attractors persists under $C^1$-generic perturbations.

\begin{theorem}\label{thmAttractor}
  Let $\Lambda$ be a transitive attractor not satisfying shadowing property then there exists $\ep_0 > 0$ such that for any $y_0 \in D(\Lambda)$ there exists $d_0 > 0$ such that for any $d < d_0$ the probability of a $d$-pseudotrajectory starting at $y_0$ to be $\ep_0$-shadowed is 0:
  $$
  p(y_0, d, \ep_0) = 0.
  $$
\end{theorem}
\begin{remark}
  Note that in Theorem \ref{thmTransitive} the choice of $d$ was uniform with respect to $y_0$ and hence Theorem \ref{thmTransitive} does not follow from Theorem \ref{thmAttractor}.
\end{remark}

\begin{proof}
Similarly to Lemma \ref{lemShFin} there exists $\ep> 0$ such that for any $d> 0$ there exists $N>0$ and a $d$-pseudotrajectory $\{p_n\}_{n = 0}^{N}$ which cannot be $\ep$-shadowed.

Fix $\ep_0 = \ep/4$, $V = B(\ep_0, \Lambda) \cap U$ and $W$, satisfying \eqref{eqAdd6.1}.

Fix $y_0 \in D(\Lambda)$. There exists $n_0$ such that $f^{n_0}(y_0) \in W$. Take $d_0 < \ep/4$ such that
\begin{enumerate}
  \item for any $\{y_n\}_{n = 0}^{n_0} \in \Omega(y_0, d_0)$ the inclusion
  \begin{equation}\label{eq.1.0.5}
    y_{n_0} \in W
  \end{equation}
  holds;
  \item $B(d_0, \overline{f(W)}) \subset W$.
\end{enumerate}
Then for any $d < d_0$ and $\{y_n\} \in \Omega(y_0, d)$ the inclusions
\begin{equation}\label{eq1.1}
  y_n \in W, \quad \mbox{for } n \geq n_0
\end{equation}
holds.

Fix arbitrarily $d < d_0$ and find $N > 0$ and a $d$-pseudotrajectory $\{p_n\}_{n = 0}^N \subset \Lambda$ which cannot be $\ep$-shadowed.

Take $\delta < d/4$ such that inclusion \eqref{eqAdd4.1} holds. Due to compactness of $M$ and hence compactness of $\overline{W}$ there exists $S > 0$ such that
$$
f^n(z) \in B(\delta/4, \Lambda), \quad n \geq S, z \in W.
$$
Arguing similarly to the proof of Theorem \ref{thmTransitive} there exists $K > 0$ such that for any $z_0 \in W$ there exists $0 < n_1 < n_2 < K$ such that
$$
\dist(f^S(z_0), f^{n_1}(r)) < \delta/2, \quad \dist(f^{n_2}(r), p_0) < \delta/2.
$$
Hence the sequence
$$
\{q_n\}_{n = 0}^{S+n_2 - n_1 + N} = \{ z_0, f(z_0), \dots, f^{S-1}(z_0), f^{n_1}(r), \dots, f^{n_2 - 1}(r), p_0, \dots, p_N \}.
$$
Define $\eta$ by \eqref{eqAdd4.2} and $L = S+K+N+1$. Arguing similarly to the proof of Theorem \ref{thmTransitive} we conclude that
$$
P^{z_0, d}(z_n \in B(\delta, q_n), n \in [0, S + n_2 - n_1 + N]) \geq \eta^{n_2 - n_1 + N} \geq \eta^{K+S+N}.
$$
Again arguing similarly to Theorem \ref{thmTransitive} and using inclusion \eqref{eq1.1} we conclude that for any $z_0 \in W$ the equality
$$
P^{z_0, d}(\{z_n\} \mbox{ is $\ep$-shadowable}) = 0
$$
holds. Combining the latter with inclusion \eqref{eq.1.0.5} we conclude that
$$
P^{y_0, d}(\{y_n\} \mbox{ is $\ep$-shadowable}) = 0,
$$
which completes the proof of Theorem \ref{thmAttractor}.
\end{proof}

\textbf{Acknowledgement.} Research of the author is supported by the Russian Science Foundation grant 14-21-00035, Saint-Petersburg State University research grant 6.38.223.2014, and RFBR 15-01-03797a.

\end{document}